\documentclass[9pt,shortpaper,twoside,web]{ieeecolor}
\pdfoutput=1
\usepackage[pdftex]{graphicx}
\usepackage[cmex10]{amsmath}
\usepackage{generic}
\usepackage{ulem}
\usepackage{subfigure}  
\usepackage{booktabs}
\usepackage{makecell}
\usepackage{amssymb}
\usepackage{tikz}
\usepackage{calc}
\usepackage{bm}
\usepackage{multirow}
\usepackage{algorithmic}
\usepackage{array}
\usepackage[caption=false,font=normalsize,labelfont=sf,textfont=sf]{subfig}
\usepackage{textcomp}
\usepackage{stfloats}
\usepackage{url}
\usepackage{verbatim}
\usepackage{threeparttable}
\usepackage[justification=centering]{caption}
\usepackage{balance}
\usepackage{comment}
\usepackage{makecell}
\usepackage{epstopdf}
\usepackage{float}
\usepackage{cite}
\usepackage{setspace}
\usepackage{ulem}

\usepackage{amsmath,amsthm}

\theoremstyle{definition}
\newtheorem{theorem}{Theorem}
\newtheorem{lemma}{Lemma}

\theoremstyle{definition}
\newtheorem{definition}{Definition}
\newtheorem{condition}{Condition}
\newtheorem{remark}{Remark}

\newtheorem{assumption}{\textit{Assumption}}
\newboolean{showcomments}
\setboolean{showcomments}{false}
\newcommand{\zjwang}[1]{\ifthenelse{\boolean{showcomments}}
	{ \textcolor[rgb]{1,0,1}{(ZW:  #1)}}{}}
\newcommand{\jxwang}[1]{\ifthenelse{\boolean{showcomments}}
	{ \textcolor{red}{(FL:  #1)}}{}}

\allowdisplaybreaks

\begin{document}
	\title{ On the Equilibrium of a Class of Leader-Follower Games with Decision-Dependent Chance Constraints }
	
	\author{
		Jingxiang Wang, Zhaojian Wang, Bo Yang, Feng Liu, and Xinping Guan \IEEEmembership{Fellow,~IEEE}
		\thanks{This work was partially supported by the National Natural Science Foundation of China (62103265), the Cast of China Association for Science and Technology (YESS20220320). (\textit{Corresponding author: Zhaojian Wang})}
		\thanks{J. Wang, Z, Wang, B. Yang, and X. Guan are with the Key Laboratory of System Control, and Information Processing, Ministry of Education of China, the Department of Automation, Shanghai Jiao Tong University, Shanghai, China (e-mail: wjx6246006@sjtu.edu.cn, wangzhaojian@sjtu.edu.cn, bo.yang@sjtu.edu.cn, xpguan@sjtu.edu.cn).}
		\thanks{F. Liu is with the State Key Laboratory of Power System and the Department of Electrical Engineering, Tsinghua University, Beijing, 100084, China (e-mail: lfeng@tsinghua.edu.cn).}
	}

	\maketitle
	\begin{abstract}

	In this paper, we study the existence of equilibrium in a single-leader-multiple-follower game with decision-dependent chance constraints (DDCCs), where decision-dependent uncertainties (DDUs) exist in the constraints of followers.
    DDUs refer to the uncertainties impacted by the leader's strategy, while the leader cannot capture their exact probability distributions. 
    To address such problems, we first use decision-dependent ambiguity sets under moment information and Cantelli's inequality to transform DDCCs into second-order cone constraints. This simplifies the game model by eliminating the probability distributions. 
    We further prove that there exists at least one equilibrium point for this game by applying  Kakutani's fixed-point theorem. Finally, a numerical example is provided to show the impact of DDUs on the equilibrium of such game models.
\end{abstract}


\section{Introduction}
 The single-leader-multiple-follower game is an extended form of the leader-follower game, 
 which is first introduced by Stackelberg \cite{stackelberg1952theory}.
 The single-leader-multiple-follower game is widely used in various disciplines, such as economy \cite{tharakunnel2009single}, energy \cite{8031035} and transportation \cite{XI2022}.
 In this game, once the leader selects an optimal strategy based on the predicted optimal responses of the followers, followers will develop their strategies accordingly to achieve Nash equilibrium among followers. Therefore, this game is also called Stackelberg-Nash game \cite{7956147}. 

In recent years, the presence of uncertainties in leader-follower games has become a significant focus in academic research.
The existing works are divided into three categories according to the treatment of uncertainties, i.e., the stochastic programming \cite{xu2005mpcc, demiguel2009stochastic}, the robust optimization \cite{aghassi2006robust, hu2013existence, wang2024existence} and the distributionally robust optimization \cite{liu2018distributionally, singh2017distributionally, fabiani2023distributionally}.  Stochastic programming aims to solve the equilibrium of the game under the condition that the probability distribution of uncertainty is known. This can be achieved through Monte Carlo simulation, sample average approximation, etc. In \cite{xu2005mpcc}, a stochastic leader-follower game model is proposed, where the demand function of the leader follows a continuous distribution. 
In addition, the implicit numerical integration method is used to estimate the expected demand of the leader, and the conditions for the existence of equilibrium in this game are also provided. In \cite{demiguel2009stochastic}, the above model is extended to the case of multiple leaders, and the unique extension of the equilibrium is proven.
Robust optimization requires players to optimize according to the worst values of uncertainties to reduce risk.
 Note that the robust optimization is distribution-free, i.e., the player has no information other than the range of uncertainties.
In \cite{aghassi2006robust}, the robust-Nash equilibrium, which combines the Nash equilibrium and robust optimization, is first proven.
This game is extended to the robust multi-leader-single-follower game in \cite{hu2013existence},  robust leader-follower game with l-type sets in \cite{wang2024existence}.
 The leader-follower game based on the distributionally robust optimization refers to players having incomplete probability distributions about the uncertainties. Players are required to select strategies based on the worst-case scenarios due to incomplete probability distributions.
  The existence of distributionally robust equilibrium is proved for Nash game and leader-follower game in \cite{liu2018distributionally}.  
The existence and characterization of mixed strategy Nash equilibrium in distributionally robust chance-constrained games are proven in \cite{singh2017distributionally}.
In \cite{fabiani2023distributionally}, the previous work is extended to the generalized Nash game with distributionally robust joint chance constraints, and the probability distribution of uncertainties is restricted to a Wasserstein Ball.

In the above works,  the ranges or probability distributions of uncertainties are not impacted by the strategies of leaders.
These uncertainties are referred to as decision-independent uncertainties (DIUs). In the real world, some uncertainties are impacted by strategies, i.e., decision-dependent uncertainties (DDUs).  
For example, the electricity demand response strategy in the current time period will change the uncertainty set of the deferred load in the next time period \cite{9983834}; current emission reduction strategies of greenhouse gas will change the probability distribution of emission reduction costs over the next period \cite{webster2012approximate};  uncertain charging demands for electric vehicles are associated with dynamic electricity pricing strategies \cite{10506707}. Therefore, it is necessary to consider DDUs in the leader-follower game. 
In \cite{zhang2022nash}, DDU is first introduced into the leader-follower game, and the existence of equilibrium is proven. In this game model, the possible value of DDU remains in an uncertainty set impacted by the leader's strategy. This extends the traditional static uncertainty set to the dynamic uncertainty set, which is parameterized in the leader's strategy.
This method essentially uses improved robust optimization to model DDU, where the ambiguity set is assumed to capture all possible values of DDU. 
However, when the extreme values of DDU have a low probability but deviate far from normal values, the ambiguity set will be very large. Then, the result may be very conservative under the worst-case scenario \cite[Page 5-7]{xie2024distributionally}.
To this end, another approach emerges, where the leader's strategy impacts the probability distribution of DDU. Although the real probability distribution is unknown, it belongs to an ambiguity set of probability distributions determined by the leader's strategy. This poses significant challenges for proving the existence of equilibrium due to the introduction of probability. 
To the best of our knowledge, there is a lack of attention focusing on the modeling and equilibrium of this game in the existing literature.

In this paper, we propose single-leader-multiple-follower games with decision-dependent chance constraints (DDCCs), where DDUs exist in the constraints of followers and the probability distributions are related to the leader's strategy.
We combine distributionally robust optimization and chance-constrained programming to model DDUs and prove the equilibrium of the game. The main contributions are summarized as follows. 
\begin{enumerate}
    \item Single-leader-multiple-follower games with DDCCs are formulated, where the probability distributions of uncertainties in the lower level are impacted by the leader's strategy. Then, the decision-dependent ambiguity sets are introduced to characterize the estimated probability distributions of DDUs. Finally, Cantelli's inequality is utilized to reformulate DDCCs as second-order cone constraints. This simplifies the game model by eliminating the probability distributions, paving the way for the equilibrium analysis.
    
    \item  The existence of the equilibrium of the single-leader-multiple-follower game with DDCCs is proved. 
    First, we construct the strategy spaces and optimal responses of the players as set-valued maps. Then we prove the properties of these set-valued maps including continuity, compactness, and convexity. In particular, different from existing literature \cite{zhang2022nash, nishimura2009robust, jia2015existence}, we prove that the strategy space of a follower is a continuous set-valued map instead of just assuming it. Finally, we prove the existence of the game equilibrium based on Kakutani's fixed-point theorem. 

\end{enumerate}

The remainder of this paper is organized as follows. Section \uppercase\expandafter{\romannumeral2} provides the notations and preliminaries. Section \uppercase\expandafter{\romannumeral3} introduces the framework for leader-follower games with DDCCs.
 Section \uppercase\expandafter{\romannumeral4} constructs set-valued maps for the leader-follower game and provides the equilibrium existence theorem with the proof. Section \uppercase\expandafter{\romannumeral5} provides an illustrative example of the equilibrium existence.
  The conclusion is given in Section \uppercase\expandafter{\romannumeral6}.



\section{Notations and Preliminaries}
This section provides the notations and preliminaries used in the rest of the paper.

\subsection{Notations}  

In this paper, the probability distribution of the random variable $x$ is denoted by $PD_x$.
The expectation of the random variable $x$ is denoted by $\mathbb{E}_{\mathbb{P}}\left[x\right]$. 
The probability measure of a random event holding is denoted by
$\mathbb{P}\left[\cdot\right]$.
The m-dimensional Euclidean space is denoted by $\mathbb{R}^m$. In the context of Euclidean space, uppercase letters, such as $X$, are utilized to denote sets. We use $X \times Y$ to denote the Cartesian product between sets $X$ and $Y$. For Euclidean space $\prod_{i=1}^I Y_i$ and $i \in\{1,2, \ldots, I\}$, a point in $Y_i$ is denoted by $y_i$,  $y_{-i} \triangleq\left(y_1, \ldots, y_{i-1}, y_{i+1}, \ldots, y_I\right)$ and $Y_{-i} \triangleq Y_{1} \times Y_{2} \times \cdots Y_{i-1} \times Y_{i+1} \times \cdots Y_{I}$.
Let  $G: Y \rightrightarrows X$ be a set-valued map from $Y$ to $X$, and the value of the set-valued map at $y$ is denoted by $G(y)$. For the Euclidean space $X$, a sequence in $X$ is denoted by $\left\{x_j\right\}$, and $j$ is the index of the sequence elements.

\subsection{Preliminaries}  
In this subsection, we introduce the definitions and the key lemmas for developing theoretical results in this paper.

\begin{definition} \label{marginal} (Definition of marginal functions and marginal maps \cite[Page 51]{jaubin1984differential})	
Let $G: Y \rightrightarrows X$ be a set-valued map and $f(x,y)$ be a real-valued function that is defined on $X \times Y$. 
The marginal function is defined as
\begin{equation}
	V(y) \triangleq \sup _x f(x, y), \text { s.t. } x \in G(y)
	\nonumber 
\end{equation}
The marginal map  $M: Y \rightrightarrows X$ is defined as
\begin{equation}
	M(y)\triangleq\{x \in G(y): f(x, y)=V(y)\}
	\nonumber 
\end{equation}
\end{definition}

Then we introduce the following continuity theorem on marginal functions and marginal maps.

\begin{lemma}\label{lemma2}
(Continuity theorem of marginal functions and marginal maps \cite[Page 51-53]{jaubin1984differential})	
For the marginal function $V(y)$ and the marginal map  $M: Y \rightrightarrows X$, which are defined in Definition \ref{marginal}, we have the following conclusion.

(\romannumeral 1 ) Suppose that $f(x,y)$ is lower semicontinuous on $X \times Y$, and the set-valued map $G$ is lower semicontinuous. Then the marginal function $ V(y) $ is lower semicontinuous.

(\romannumeral 2 )  Suppose that $f(x,y)$ is upper semicontinuous on $X \times Y$, and the set-valued map $G$ is upper semicontinuous  with compact values. Then the marginal function $	V(y) $ is upper semicontinuous.

(\romannumeral 3 )  Suppose that $f(x,y)$ is continuous on $X \times Y$, and the set-valued map $G$ is continuous  with compact values. Then the marginal map $	M(y) $ is upper semicontinuous.
\end{lemma}

\begin{lemma}\label{lemma1}
	(Cantelli's inequality \cite{cantelli1929sui})
Let $\xi$ be a random variable. The expectation and variance of the random variable $\xi$ is denoted by $\mathbb{E}_{\mathbb{P}}[\xi]$ and $\sigma^2$, respectively.  Then, for $\forall \lambda>0$, we have
	\begin{equation}
		\operatorname{\mathbb{P}}\left[\xi-\mathbb{E}_{\mathbb{P}}[\xi] \geq \lambda\right] \leq \frac{\sigma^2}{\sigma^2+\lambda^2}
		\nonumber 
	\end{equation}
\end{lemma}

\begin{lemma}\label{lemma3}
	(Kakutani's fixed-point theorem \cite{kakutani1941generalization})	
Let $X$ be a non-empty, compact and convex subset in Euclidean space $\mathbb{R}^n$. If set-valued map $K: X \rightrightarrows X$ is upper semicontinuous  with nonempty compact convex values, and there exists fixed point $x^* \in X$ such that $x^* \in K\left(x^*\right)$.
\end{lemma}

\section{Game Model and DDCC Reformulation}

In this section, we first formulate a general single-leader-multi-follower game model. Different from the existing leader-follower game,  DDUs exist in the lower-level game, whose probability distributions depend on the decision variables in the upper-level game. Since DDUs exist in constraints, we describe these special constraints by chance-constrained programming, i.e., DDCCs.
Then, we transform these DDCCs into second-order cone constraints using the improved distributionally robust method.

\subsection{Game model} 
\begin{figure}[t]
	\centering
	\includegraphics[width=0.45  \textwidth]{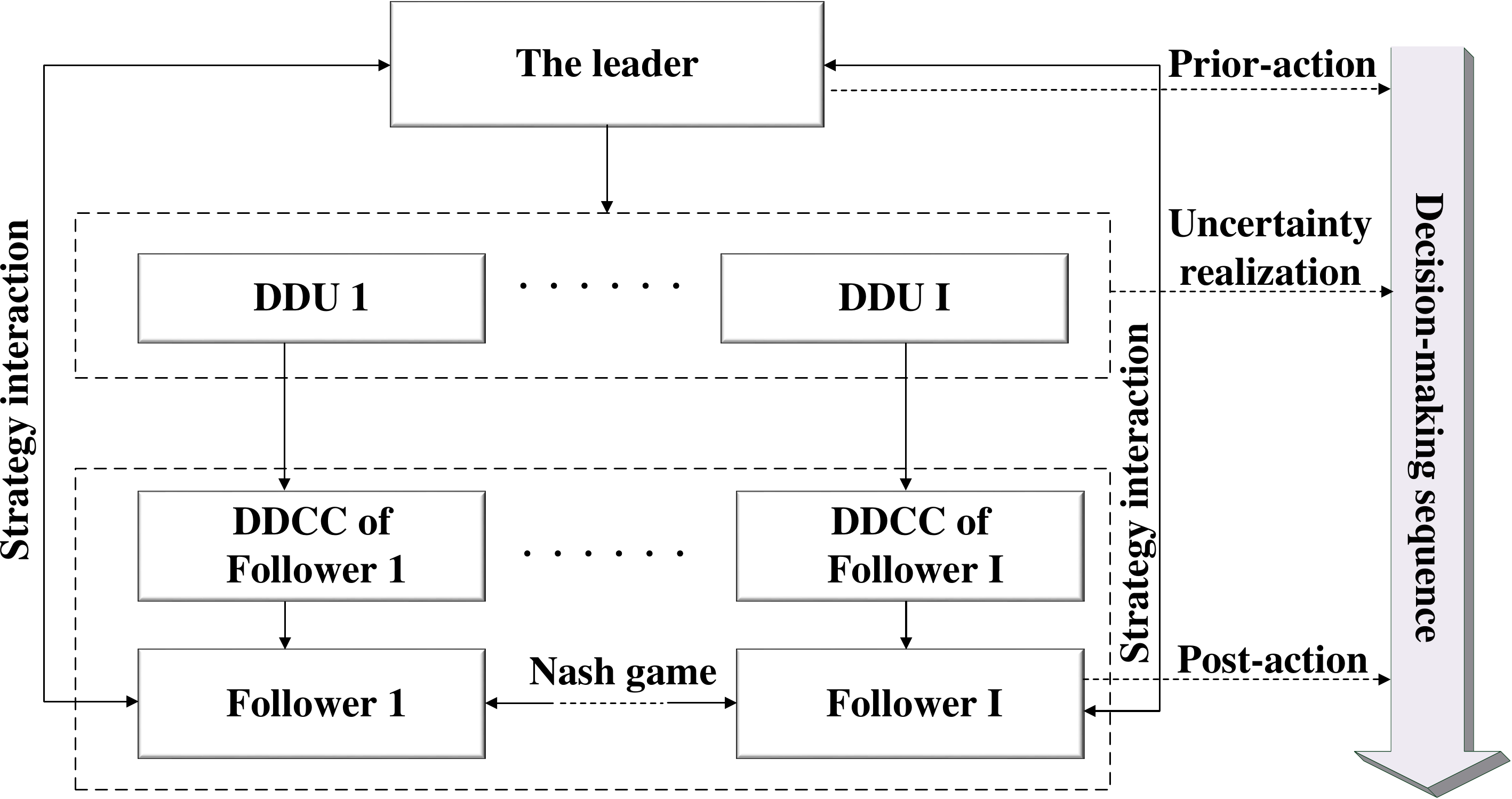}
	\caption{The framework of leader-follower game with DDCCs}
	\label{Framework}
\end{figure}
In this subsection, we formulate the single-leader-multi-follower game model with DDCCs.
Fig.\ref{Framework} provides the framework of the leader-follower game with DDCCs.
For this type of game model, the following four features need to be taken into account:

(\romannumeral 1 ) \textbf{Multi-level structure}: There is a single leader in the upper level and multiple followers in the lower level. Each player has their own strategy set as well as a payoff function.
Players are all selfish individuals, which means that this game is a non-cooperative game.
Due to the multi-level structure, the leader selects a strategy first and followers select strategies later.

(\romannumeral 2 ) \textbf{Asymmetric information}: 
The information between the upper level and lower level is asymmetric. The leader knows the payoff functions and strategy sets of followers. One follower can only observe the strategies of the leader and other followers.

(\romannumeral 3 ) \textbf{DDUs}:
There are DDUs in the constraints of followers. The leader only has part of the reference historical data of the uncertainties, i.e., the leader only has incomplete information about these uncertainties.
In addition, the probability distributions of DDUs are dependent on the strategy of the leader.

(\romannumeral 4 ) \textbf{Robustness}: The leader is risk-averse to DDUs. Therefore, the leader's best response is made in the worst probability distributions of the DDUs. 

Based on the above features, we construct the mathematical form of this game model as a multi-tuple, i.e.,
\begin{equation}
	\Gamma \triangleq\left\{x, X, f(x, \boldsymbol{y}), N, y_i, \Omega_i, \varphi_i\left(x, y_{-i}, y_{i}\right), \xi_i(x), D_i(x)\right\}
\end{equation}
Next, we will explain these elements in detail.

\begin{definition}
For the multi-tuple $\Gamma$, each element is defined as follows.

(\romannumeral 1 ) \textbf{Uncertainties}:  For any $ i \in N$,  $\xi_i(x)$ is the one-dimensional uncertainty variable. 
The exact probability distribution of DDU $\xi_i(x)$ is difficult to obtain. Furthermore, $\xi_i(x)$ depends on the strategy $x$ which is determined by the leader.
Therefore, we introduce a decision-dependent ambiguity set $D_i(x)$ to characterize the plausible probability distribution of DDU $\xi_i(x)$.  
$D_i(x)$ characterizes an implicit relationship between the probability distribution of the DDU $\xi_i(x)$ and the strategy $x$. We refrain from discussing the specific form of $D_i(x)$ here, as the construction process of $D_i(x)$ will be elaborated in Section \ref{ambiguity set}.

(\romannumeral 2 ) \textbf{Lower level}:  $N \triangleq\{1,2,3, \ldots, I\}$ is the set of followers. 
$y_{i} \in Y_i \subseteq \mathbb{R}$ is the strategy of the follower $i$.
The strategy space $\Omega_i \subseteq \mathbb{R}$ is defined as follows.
\begin{equation}
	\Omega_i=\left\{y_i \in Y_i \mid \xi_i(x) \cdot y_i \leq b_i\right\} 
\end{equation}
where $\xi_i(x)$ is the DDU, $b_i$ 
is a constant, and $x \in X \subseteq \mathbb{R}^m$ is the strategy of the leader. The lower-level game is defined as a Nash game. The payoff function of each follower is dependent on other player's strategies $x$ and $y_{-i}$, and each follower cannot determine the strategies $x$ and $y_{-i}$. The payoff function is defined as $\varphi_i: X \times Y_{-i} \times Y_{i} \rightarrow \mathbb{R}$. Follower $i$ wants to minimize the payoff function $ \varphi_i\left(x, y_{-i}, y_{i}\right)$. Therefore, the lower-level game is characterized as follows.
	\begin{align}
		& \min_{y_i} \varphi_i\left(x, y_{-i}, y_i\right) ,\  \text { s.t. } y_i \in \Omega_i
	\end{align}

(\romannumeral 3 ) \textbf{Upper level}: There is a single leader with strategy $x \in X$. The payoff function of the leader is defined as $f: X \times \prod_{i=1}^I Y_i \rightarrow \mathbb{R}$.
 $f(x, \boldsymbol{y})$ is single-valued where  $\boldsymbol{y} \triangleq\left(y_1, \ldots, y_i, \ldots, y_I\right) \in \prod_{i=1}^I Y_i\subseteq \mathbb{R}^I $ is the strategies of followers.
It is worth noting that the leader cannot determine the strategies $\boldsymbol{y}$. The leader wants to maximize its payoff function $f(x, \boldsymbol{y})$. 
In addition, since the leader is risk averse, we can use decision-dependent ambiguity set $D_i(x)$ as well as chance-constrained programming to describe the uncertainty constraints in $\Omega_i$.
The reconstructed constraint, i.e., DDCC, is reformulated as follows.
\begin{equation}\label{DDCC1}
\inf _{PD_{\xi_i(x)} \in D_i(x)} {\mathbb{P}}\left[
\xi_i(x) \cdot y_i \leq b_i
\right]\geq 1-\alpha_i 
\end{equation}
where $\alpha_i $ denotes the risk tolerance of constraint violation for the leader. Generally, $\alpha_i $ is very small, e.g., $\alpha_i=0.05$. 
Therefore, the upper-level game is characterized as follows.
\begin{subequations} \label{leader}
	\begin{align}
		& \max_x f(x, \boldsymbol{y}) \\
		& \text { s.t. } x \in X \\
         &\quad y_i \in \arg \min _{y_i} \Big\{ \varphi_i \big( x, y_{-i}, y_i \big) : \notag \\
		& \big\{  Y_i \mid \inf _{PD_{\xi_i(x)} \in D_i(x)} {\mathbb{P}}\left[
\xi_i(x) \cdot y_i \leq b_i
\right]\geq 1-\alpha_i \big\} \Big\} , \forall i \in N
	\end{align}
\end{subequations}
\end{definition}

\subsection{Definition of the equilibrium} 

In this subsection, we define the equilibrium of game $\Gamma$. First, we introduce Condition \ref{condition for equilibrium} which is the key to defining equilibrium.

\begin{condition} \label{condition for equilibrium}
For the given $\alpha_i \in [0,1]$, consider a point $(x^*,\boldsymbol{y}^*)$ with $x^* \in X$ and 

$y_i^* \in \left\{Y_i \mid \inf _{PD_{\xi_i(x^*)} \in D_i(x^*)} {\mathbb{P}}\left[
\xi_i(x^*) \cdot y_i \leq b_i
\right]\geq 1-\alpha_i \right\}$.

(a) For the leader, $f\left(x^*, \boldsymbol{y}^*\right) \geq f\left(x^1, \boldsymbol{y}^*\right)$ holds for any $x^1 \in X$.

(b) For each follower $i$, $\varphi_i\left(x^*, y_{-i}^*, y_i^*\right) \leq \varphi_i\left(x^*, y_{-i}^*, y_i^1\right)$ holds for any 

$y_i^1 \in \left\{Y_i \mid \inf _{PD_{\xi_i(x^*)} \in D_i(x^*)} {\mathbb{P}}\left[
\xi_i(x^*) \cdot y_i \leq b_i
\right]\geq 1-\alpha_i \right\}$.

\end{condition}

Based on Condition \ref{condition for equilibrium}, we define the equilibrium of game $\Gamma$. 
\begin{definition} \label{equilibrium}
	If $(x^*,\boldsymbol{y}^*)$ satisfies Condition \ref{condition for equilibrium} (a) and (b), then $(x^*,\boldsymbol{y}^*)$ is a equilibrium of game $\Gamma$. 
\end{definition}

In Definition \ref{equilibrium},
$x^*$ is the equilibrium strategy of the leader, which is determined under the worst probability distributions of DDUs $\xi_i(x^*)$ in $D_i(x^*)$. $\boldsymbol{y}^*$ is the equilibrium strategies of all followers, which is the Nash equilibrium of the lower-level game under strategy $x^*$. Therefore, the equilibrium of game $\Gamma$ is an extension of the Stackelberg-Nash equilibrium \cite{7956147}.

The introduction of DDU brings two challenges to analyzing the existence of the equilibrium: 1) the form of the ambiguity set $D_i(x)$ is unknown; 2) since the leader is risk averse, the DDCC under the worst probability distribution of DDU needs to be addressed. Next, we will focus on these challenges.


\subsection{Construction of the ambiguity set} \label{ambiguity set}

In this subsection, we use the moment information to construct the decision-dependent ambiguity set $D_i(x)$. First, we give the finite support assumption as follows.

\begin{assumption} \label{assumption0}
For any $x \in X$, ${PD_{{\xi}_i(x)} \in D_i(x)}$ has a common decision-independent finite support $\Big\{\widetilde{\xi}_i^k\Big\}_{k=1}^K$,
where $k$ is the index of the set,  $K$ is the size of the set, and $\widetilde{\xi}_i^k$ is the sampled value. 
\end{assumption}

Assumption \ref{assumption0} will be used throughout the rest of the paper unless otherwise stated. 
This assumption is standard in the ambiguity set construction for DDU \cite{luo2020distributionally, yu2022multistage}, which ensures that  
\begin{align*}
\mathbb{P}\left[\xi_i(x) \in\left\{\widetilde{\xi}_i^k\right\}_{k=1}^K\right]=1
	\end{align*}
Based on Assumption \ref{assumption0}, we can assume that the leader has a set of reference historical data samples $\Big\{\widetilde{\xi}_i^k\Big\}_{k=1}^K$for DDU $\xi_i(x)$. Then, we give the mean and variance of this sample set as follows.
\begin{align}
	\mu_i &=\frac{1}{K} \sum\nolimits_{k=1}^K \widetilde{\xi}_i^k\\
	\Sigma_i&=\frac{1}{K} \sum\nolimits_{k=1}^K\left(\widetilde{\xi}_i^k-\mu_i\right)^2
	\end{align}
If $K$ is large enough,  $\mu_i$ and $\Sigma_i$ will be equal to the true mean and variance of $\xi_i(x)$.
A natural idea is that the probability distribution in an ambiguity set will match the mean $\mu_i$ and variance $\Sigma_i$ in the set $\Big\{\widetilde{\xi}_i^k\Big\}_{k=1}^K$. Denote the ambiguity set by $D_{i,1}$, which is defined as follows.
\begin{equation}
	\scalebox{1}{$
	D_{i,1}=\left\{\begin{array}{ll} 
			& \mathbb{E}_{\mathbb{P}}\left[\xi_i(x)\right]=\mu_i, \\
			PD_{{\xi}_i(x)} \in \mathcal{P}\left(\mathbb{R}\right): & \\
			& \mathbb{E}_{\mathbb{P}}\left[\left(\xi_i(x)-\mu_i\right)^2\right]=\Sigma_i
		\end{array}\right\}
		$}
\end{equation}
where $\mathcal{P}\left(\mathbb{R}\right)$ denotes the set of all probability distributions on $\mathbb{R}$. 
However, $K$ is limited in practical scenarios, 
and the probability distribution of DDU $\xi_i(x)$ will change with strategy $x$.
These factors cause deviations in the mean $\mu_i$ and variance $\Sigma_i$ from the true probability distribution.
 At this point, the standard ambiguity set $\mathcal{D}_{i,1}$ is no longer suitable for the game model. 
 Inspired by the second-order moment ambiguity set \cite{delage2010distributionally}, we improve the ambiguity set $\mathcal{D}_{i,1}$ to $D_i(x)$. Specifically, we aim to use the estimated moment information from $\mathcal{D}_{i,1}$ to envelop the true variance and mean in $D_i(x)$. Based on this idea, we define  $D_i(x)$ as follows.
\begin{equation}\label{Di}
	\scalebox{0.94}{$
D_i(x)=\left\{\begin{array}{ll} 
			& \left(\mathbb{E}_{\mathbb{P}}\left[\xi_i(x)\right]-\mu_i\left(x\right)\right)^2 \leq \gamma_{i,1} \Sigma_i, \\ PD_{{\xi}_i(x)} \in \mathcal{P}\left(\mathbb{R}\right): & \\
			& \mathbb{E}_{\mathbb{P}}\left[\left(\xi_i(x)-\mu_i\left(x\right)\right)^2\right] \leq \gamma_{i,2} \Sigma_i
		\end{array}\right\}
		$}
\end{equation}

In the decision-dependent ambiguity set $D_i(x)$, $\gamma_{i,1}>0$ and $\gamma_{i,2}>1$ are two given parameters and the estimated mean $\mu_i\left(x\right)$ depends on the strategy $x$. The first constraint ensures that the true mean lies
in a circle centered on the estimated mean $\mu_i\left(x\right)$. The second constraint ensures that the true variance is limited to 
$\gamma_{i,2} \Sigma_i-\left(\mathbb{E}_{\mathbb{P}}\left[\xi_i(x)\right]-\mu_i(x)\right)^2$. 
It is worth noting that $\mathbb{E}_{\mathbb{P}}\left[\left(\xi_i(x)-\mu_i(x)\right)^2\right]=\mathbb{E}_{\mathbb{P}}\left[\left(\xi_i(x)-\mathbb{E}_{\mathbb{P}}\left(\xi_i(x)\right)\right)^2\right]+\left(\mathbb{E}_{\mathbb{P}}\left[\xi_i(x)\right]-\mu_i(x)\right)^2$.
In fact, $\gamma_{i,1}$ and $\gamma_{i,2}$  are parameters of a risk measurement method. They measure the leader's confidence in the estimated mean $\mu_i\left(x\right)$ and the estimated variance $\Sigma_i$, respectively \cite{delage2010distributionally}. We consider the impact of strategy $x$ on the probability distribution of the DDU $\xi_i(x)$, which is mainly reflected in the estimated mean $\mu_i\left(x\right)$. Therefore, $D_i(x)$ is a method for constructing a statistical distance between the probability distribution of DDU $\xi_i(x)$ and the reference historical data sample $\Big\{\widetilde{\xi}_i^k\Big\}_{k=1}^K$.
\subsection{DDCC reformulation} 

Although we have completed the construction of the ambiguity set $D_i(x)$,  the DDCC under the worst probability distribution of DDU still needs to be addressed.
In this subsection, we reformulate DDCC \eqref{DDCC1} with the decision-dependent ambiguity set $D_i(x)$ to the second-order cone form, and the reformulation is carried out under the worst probability distribution.
We first present the reformulation result as follows.
\begin{theorem}
	\label{theorem1}
 The DDCC \eqref{DDCC1}  with the decision-dependent ambiguity set $D_i(x)$ is equivalent to 
	\begin{equation} \label{eu1}
		\mu_i(x)y_i+l_i \sqrt{\Sigma_i y_i^2} \leq b_i
	\end{equation}
	where 
	\begin{align*}
		l_i=\begin{cases}\sqrt{\gamma_{i,1}}+\sqrt{\left(\frac{1-\alpha_i}{\alpha_i}\right)\left(\gamma_{i,2}-\gamma_{i,1}\right)},
			& \text { when } \gamma_{i,1} / \gamma_{i,2} \leq \alpha_i \\ 
			\sqrt{\frac{\gamma_{i,2}}{\alpha_i}},
			& \text { when } \gamma_{i,1} / \gamma_{i,2}>\alpha_i \end{cases}
	\end{align*}
\end{theorem}

The proof of Theorem \ref{theorem1} is provided in Appendix \ref{section_appendix_theorem1}. We introduce auxiliary variables 
\begin{subequations} \label{auxiliary}
\begin{align} \label{auxiliary1}
		\widetilde{s}_i&=\xi_i(x)-\mu_i(x),\\
	\label{auxiliary2}	\widetilde{\beta}_i&=\widetilde{s}_iy_i,\\
 \label{auxiliary3}
		c_i&=b_i-\mu_i(x)y_i
	\end{align} 
 \end{subequations}
 to decompose DDCC \eqref{DDCC1} into an inner-outer two-level problem. We further use Lemma \ref{lemma1} to solve the worst probability boundary. Finally, we obtain the second-order cone form \eqref{eu1} of DDCC \eqref{DDCC1}, which means that strategy space ${\Omega}_i$ has changed. We provide the reformulated form of the game $\Gamma$ as follows.
\begin{subequations} \label{follower2}
	\begin{align}
		\label{l11}
		& \max _x f(x, \boldsymbol{y}) \\
		\label{l12}
		& \text { s.t. } x \in X \\
		\label{l13}
		& \quad \boldsymbol{y} \triangleq \big( y_1, \ldots, y_i, \ldots, y_I \big) \\
		& \quad y_i \in \arg \min _{y_i} \Big\{ \varphi_i \big( x, y_{-i}, y_i \big) : \notag \\
		\label{follower22}
		& \qquad \widehat{\Omega}_i = \big\{ y_i \in Y_i \mid \mu_i(x) y_i + l_i \sqrt{\Sigma_i y_i^2} \leq b_i \big\} \Big\}, \forall i \in N
	\end{align}
\end{subequations}

It is worth noting that game model  $\Gamma$ \eqref{follower2} is obtained under the worst probability distributions of DDUs, which satisfies the risk aversion of the leader. Different from the game model with DIUs
\cite{aghassi2006robust, hu2013existence, wang2024existence, liu2018distributionally, singh2017distributionally, fabiani2023distributionally},
 the strategy space $\widehat{\Omega}_i$ of $y_i$ is impacted by strategy $x$, which presents a challenge to the existence of equilibrium in game $\Gamma$.
\begin{remark}[DDCC] DDCC is an important branch of the DDU problem. There already exists some research on DDCC in recent works \cite{ 10252038, basciftci2021distributionally,10038580}. 
In fact, this form of uncertainty constraint is common in practical problems, e.g.,
uncertainties in the management of hybrid energy \cite{10252038},
uncertainties in the facility location \cite{basciftci2021distributionally}, and uncertainties in the risk management \cite{10038580}. Therefore, when the finite support assumption (Assumption \ref{assumption0}) holds, our method provides a generalized solution to these problems.
\end{remark}

\section{Existence of Equilibrium for Leader-Follower Games with DDCCs}

To prove the existence of the equilibrium point in game $\Gamma$, we first construct the game model as multiple set-valued maps and provide the key assumptions.
We then prove some properties of set-valued maps, including convexity, compactness, and continuity. 
Finally, we introduce Kakutani's fixed-point theorem to prove the existence of the equilibrium point.

\subsection{Construction of set-valued maps} 

In this subsection, we mainly construct set-valued maps in the game $\Gamma$, which will be used to analyze the existence of fixed points. 
Before constructing set-valued maps for the game $\Gamma$, we analyze the strategy spaces $\widehat{\Omega}_i$ of followers in \eqref{follower2}. 
We have known that $\widehat{\Omega}_i$ is changing with $x$ since $\mu_i\left(x\right)$ is a real-valued function with respect to $x$. Then we can construct  $\widehat{\Omega}_i$ as a set-valued map $S_{F,i}$:$X\rightrightarrows Y_{i}$ with the value $S_{F, i}(x)$, i.e.,
\begin{equation} \label{follower3}
	\begin{aligned}
		S_{F, i}(x)  \triangleq\left\{y_i \in Y_{i} \mid 	\mu_i\left(x\right)y_i+l_i \sqrt{\Sigma_i {y_i}^2} \leq b_i\right\}
	\end{aligned}
\end{equation}

Denote by $H_i$: $X \times Y_{-i} \rightrightarrows Y_{i}$ the best response of follower $i$,
and construct the set-valued map value $H_i\left(x, y_{-i}\right)$ as follows.
\begin{equation}\label{Hi1}
	\begin{aligned}
		& H_i\left(x, y_{-i}\right) \triangleq\left\{y_i \in S_{F, i}(x) \mid \right. \\
		& \left.\varphi_i\left(x, y_{-i}, y_i\right) \leq \varphi_i\left(x, y_{-i}, y_i^1\right), \forall y_i^1 \in S_{F, i}(x)\right\}
	\end{aligned}
\end{equation}

The lower-level game \eqref{follower22} is a Nash game.
Denote by  $B$: $X \times \prod_{i=1}^I Y_{i} \rightrightarrows \prod_{i=1}^I Y_{i}$ the best response of all followers, and construct the set-valued map value $B(x,\boldsymbol{y})$ as follows.
\begin{equation}\label{B1}
	B(x,\boldsymbol{y}) \triangleq \prod\nolimits_{i=1}^I H_i\left(x,y_{-i}\right)
\end{equation}

For upper-level game \eqref{l11}-\eqref{l13}, we construct the best response set of the leader as follows.
\begin{equation}\label{K1}
	\begin{aligned}
		K=  \{x \in X&\mid x \in \arg \max _x \max _{\boldsymbol{y}} f(x, \boldsymbol{y}) \\
		& \qquad\qquad\text { s.t. } x \in X, \boldsymbol{y} \in B(x, \boldsymbol{y})\}
	\end{aligned}
\end{equation}

Since the leader determines strategy $x$ before followers, $\boldsymbol{y}$ depends on $x$, and $K$ does not depend on $\boldsymbol{y}$. Based on Definition \ref{marginal},  we construct the marginal function as follows.
\begin{equation}\label{VX}
	V(x)=\max _{\boldsymbol{y} \in B(x, \boldsymbol{y})} f(x, \boldsymbol{y})
\end{equation}

Then, the best response set $K$ of the leader is as follows.
\begin{equation}\label{K2}
	K=\left\{x \in X \mid x \in \arg \max _x V(x) \text { s.t. } x \in X\right\}
\end{equation}

\subsection{Existence of the game equilibrium} 
In this subsection, we first give two assumptions which are sufficient conditions for the existence of equilibrium in game $\Gamma$.
\begin{assumption} \label{assumption1}
	For the upper-level game \eqref{l11}-\eqref{l13}:
	
	(\romannumeral 1 ) The payoff function $f(x, \boldsymbol{y})$ is continuous  with respect to $(x, \boldsymbol{y}) \in X \times \prod_{i=1}
^I Y_i$.
	
	(\romannumeral 2 ) The leader's strategy space $X$ is a non-empty, compact set.
\end{assumption}

\begin{assumption} \label{assumption2}
	For the lower-level game \eqref{follower22}:
	
	(\romannumeral 1 ) For any $x \in X$ and 
	$y_{-i} \in Y_{-i}$, the payoff function $\varphi_i\left(x, y_{-i}, y_i\right)$ of the follower $i$ is convex.  The payoff function is continuous with respect to $ \left(x, y_{-i}, y_i\right) \in X \times \prod_{i=1}^I Y_i$.
	
	(\romannumeral 2 ) For any $ i \in N$, $\mu_i\left(x\right)$ is continuous with respect to $x \in X$, and the constant $b_i>0$. 
	
	(\romannumeral 3 ) For any $ i \in N$, $Y_i$ is a non-empty, compact, and convex set that contains the origin. 
\end{assumption}
Assumption \ref{assumption1} and Assumption \ref{assumption2} (\romannumeral 1 ), (\romannumeral 2 ) are standard for the equilibrium proof \cite{9424958, 8998158, 10493142}. Assumption \ref{assumption2} (\romannumeral 3 ) 
is widely used in research on endogenous constraints, 
which is the crucial condition
for proving the lower semicontinuity of $S_{F,i}$ \cite[Page 343]{kreps2013microeconomic}.
Then, we have the following lemmas for game $\Gamma$.

 \begin{lemma} \label{lemmaSF}
	Suppose Assumptions \ref{assumption1} and \ref{assumption2} hold, for any $i \in N$, the set-valued map $S_{F,i}$ defined in \eqref{follower3} has the following properties.
 
	(\romannumeral 1 ) The set-valued map $S_{F,i}$ is continuous with respect to  $x \in X$;	
 
	(\romannumeral 2 ) For any $x \in X$ , $S_{F, i}(x)$ is a  compact and convex set.
\end{lemma}

\begin{proof}
	
		\item  	(\romannumeral 1 ) \textbf{Continuity}: We first prove that $S_{F,i}$ is upper semicontinuous. Let us recall the fact that $S_{F, i}$ is upper semicontinuous at $x$, whenever $\left\{x_j\right\}$ in $X$ with the limit $x \in X$, and $\left\{y_{i,j}\right\}$ is a sequence in $Y_i$ such that $y_{i,j} \in S_{F,i}(x_j)$ for all $j$, then $\lim _{j \rightarrow \infty} y_{i, j} \in  S_{F,i}(x)$ \cite[Page 471]{kreps2013microeconomic}.
  
		We can set a sequence $\left\{\left(x_j, y_{i, j}\right)\right\}$ in $X \times Y_{i}$ with $y_{i,j} \in S_{F,i}(x_j)$. Let the limit of $\left\{\left(x_j, y_{i, j}\right)\right\}$ be $(x,y_i)$.
        Based on the continuity of  $\mu_i\left(x\right)$ in Assumption \ref{assumption2} (\romannumeral 2 ),  we have
        \begin{equation}\label{xn2}
        	\lim _{j \rightarrow \infty} \mu_i\left(x_j\right) y_{i, j}+l_i \sqrt{\Sigma_i}\left|y_{i, j}\right| = \mu_i(x) y_i+l_i \sqrt{\Sigma_i}\left|y_i\right|
        \end{equation}
       Since $y_{i,j} \in S_{F,i}(x_j)$, we have $\mu_i(x) y_i+l_i \sqrt{\Sigma_i}\left|y_i\right| \leq b_i$, which means that $y_{i} \in S_{F,i}(x)$. Therefore, $S_{F,i}$ is upper semi-continuous for any  $x \in X$.
 
 Next, we prove that $S_{F,i}$ is lower semicontinuous. Let us recall the fact that $S_{F, i}$ is lower semicontinuous at $x$, if for any $\left\{x_j\right\}$ in $X$ with the limit $x \in X$, $j>J$ with some sufficiently large $J$, and $y_{i} \in S_{F,i}(x)$, we can find a sequence 
$\left\{y_{i,j}\right\}$ with $y_{i,j} \in S_{F,i}(x_j)$ such that $\left\{y_{i,j}\right\} \rightarrow y_i$.  \cite[Page 472]{kreps2013microeconomic}. 

 Based on Assumption \ref{assumption2} (\romannumeral 3 ), $Y_i$ is a non-empty, compact and convex set which contains the origin.  Then we can conclude that for any $ x \in X$, $\exists \hat{y}_i \in Y_i$ such that $\mu_i\left(x\right)\hat{y}_i+l_i \sqrt{\Sigma_i}\left|\hat{y}_i\right| < b_i$ holds strictly because of $b_i>0$. Actually, for the follower $i$, $\hat{y}_i$ is a common interior point in the set-valued map $S_{F, i}$ \cite[Page 343]{kreps2013microeconomic}.  This implies that $\exists \delta>0$, such that
 \begin{equation}\label{hat1}
 	\mu_i(x) \hat{y}_i+l_i \sqrt{\Sigma_i}\left|\hat{y}_i\right| \leq b_i-\delta
 \end{equation}
 Then we set $\left\{x_j\right\} \rightarrow x$ and $y_{i} \in S_{F,i}(x)$.
  To show lower semi-continuity, we have to produce a sequence $\left\{y_{i,j}\right\}$ with $y_{i,j} \in S_{F,i}(x_j)$ such that $\left\{y_{i,j}\right\} \rightarrow y_i$. Since $Y_i$ is convex, all the convex combinations of $y_i$ and $\hat{y}_i$ are in  $Y_i$. By fixing a large positive integer $M$, there exists $J_M$ sufficiently large such that 
  \begin{equation}\label{hat2}
  	\mu_i(x_j) y_{i}+l_i \sqrt{\Sigma_i}\left|y_{i}\right| \leq b_i+\frac{\delta}{M}, \forall j>J_M
  \end{equation}
  
  Combining \eqref{hat1} and \eqref{hat2}, we have the 
  convex combination of $y_i$ and $\hat{y}_i$, i.e.,
  \begin{align}\label{hat3}
  		& \mu_i\left(x_j\right)\left(\frac{M}{M+1} y_i+\frac{1}{M+1} \hat{y}_i\right)
  		+l_i \sqrt{\Sigma_i}\left|\frac{M}{M+1} y_i+\frac{1}{M+1} \hat{y}_i\right| \nonumber\\
  		& \leq\left(\frac{M}{M+1}\right)\left(\mu_i\left(x_j\right) y_i+l_i \sqrt{\Sigma_i}\left|y_i\right|\right)\nonumber\\
  		&
  		\qquad\qquad\qquad+\left(\frac{1}{M+1}\right)\left(\mu_i\left(x_j\right) \hat{y}_i+l_i \sqrt{\Sigma_i}\left|\hat{y}_i\right|\right) \nonumber\\
  		& \leq\left(\frac{M}{M+1}\right)\left(b_i+\frac{\delta}{M}\right)+\left(\frac{1}{M+1}\right)\left(b_i-\delta\right) \nonumber\\
  		& =b_i, \forall j >J_M
   \end{align}
 The result in \eqref{hat3} is obtained mainly through the trigonometry theorem. Then we can conclude that
 \begin{equation}\label{hat4}
 	\left(\frac{M}{M+1} y_i+\frac{1}{M+1} \hat{y}_i\right) \in S_{F, i}\left(x_j\right)
 \end{equation}
 
 Now we can choose $J_1$, $J_2$, $J_3$, ..., so that $J_1<J_2<J_3$.... For any $\forall j > J_1$, let $M_j$ be the largest integer such that $j>J_{M_j}$. This ensures that $lim _{j \rightarrow \infty} M_j=+\infty$. 
 Then we set 
 \begin{equation}\label{hat5}
 	y_{i, j}=\frac{M_j}{M_j+1} y_i+\frac{1}{M_j+1} \hat{y}_i
 \end{equation}
 Based on \eqref{hat4}, we have known that $y_{i, j} \in S_{F, i}\left(x_j\right) $. Based on \eqref{hat5}, we have $\left\{y_{i,j}\right\} \rightarrow y_i$. So the set-valued map $S_{F,i}$ is lower semicontinuous for any $x \in X$. Because the set-valued map $S_{F,i}$ is both upper semicontinuous and lower semicontinuous,  $S_{F,i}$ is continuous for any  $x \in X$ \cite[Page 472]{kreps2013microeconomic}.
 
\item  	(\romannumeral 2 ) \textbf{Compactness and convexity}:
 We have known that the set-valued map $S_{F,i}$ is upper semicontinuous. In addition, for any $x \in X$, we have $S_{F, i}(x) \subset Y_i$. This means that the set-valued map $S_{F,i}$ is upper semicontinuous and locally bounded. Therefore, for any $ x \in X$ , $S_{F, i}(x)$ is a  compact set \cite[Page 472]{kreps2013microeconomic}.
 
 Next, we prove the convexity of $S_{F, i}(x)$.
 For any $x \in X$, $\mu_i\left(x\right)y_i+l_i \sqrt{\Sigma_i {y_i}^2} \leq b_i$ is a second order cone, which is convex. The intersection of $Y_i$ and $\mu_i\left(x\right)y_i+l_i \sqrt{\Sigma_i {y_i}^2} \leq b_i$ must be convex because $Y_i$ is also convex.

 This completes the proof.
\end{proof}

 Lemma \ref{lemmaSF} shows that the set-valued map $S_{F,i}$ of follower constraint is continuous with compact and convex values. Next, we will show some key properties of set-valued map $B$, which denotes the best response of all followers.

\begin{lemma} \label{lemmaB}
	Suppose Assumptions \ref{assumption1} and \ref{assumption2} hold, the set-valued map $B$ defined in \eqref{B1} has the following properties.
	
	(\romannumeral 1 ) The set-valued map $B$ is non-empty in $X \times \prod_{i=1}^I Y_{i}$;

	(\romannumeral 2 ) For any $x \in X$ and 
	$\boldsymbol{y} \in \prod_{i=1}^I Y_{i}$, $B(x,\boldsymbol{y})$ is a convex set;
	
	(\romannumeral 3 )  For any $x \in X$ and 
	$\boldsymbol{y} \in \prod_{i=1}^I Y_{i}$, $B(x,\boldsymbol{y})$ is a compact set;
	
		(\romannumeral 4 )  The set-valued map $B$ is upper semicontinuous  with respect to $(x,\boldsymbol{y})$.
\end{lemma}
\begin{proof}
		\item  	(\romannumeral 1 ) \textbf{Non-empty}: Based on Lemma \ref{lemmaSF} (\romannumeral 2 ), for any $x \in X$, $S_{F,i}(x)$ is a convex set. For any $x \in X$ and 
		$ y_{-i} \in Y_{-i}$, the payoff function $\varphi_i\left(x, y_{-i}, y_i\right)$ of the follower $i$ is convex.  Thus for any given $(x, y_{-i})$, there exists at least one $y_i$ such that
		\begin{equation}
			\begin{aligned} & y_i \in \arg \min _{y_i^1} \varphi_i\left(x, y_{-i}, y_i^1\right) \\ & \text { s.t. } \quad y_i^1 \in S_{F, i}(x)
			\end{aligned}\nonumber 
		\end{equation}
	This means that $H_i(x,y_{-i})$ is a non-empty set. $B(x,\boldsymbol{y})$ is also a non-empty set because of $H_i(x,y_{-i}) \subset B(x,\boldsymbol{y})$. 
		\item  (\romannumeral 2 ) \textbf{Convexity}: For any given $ x \in X$ and $ y_{-i} \in Y_{-i}$, let $y_i^1, y_i^2$ be any two points of  $H_i\left(x, y_{-i}\right)$.
		For any $ y_i^1, y_i^2 \in H_i\left(x, y_{-i}\right)$ and $\gamma \in[0,1]$, let $\hat{y}_i=\gamma y_i^1+(1-\gamma) y_i^2$, $\hat{y}_i \in S_{F, i}(x)$.
 		 Based on Assumption \ref{assumption2} (\romannumeral 1 ), we have
		 \begin{equation}\label{convex1}
		 	\begin{aligned}
		 		&\varphi_i\left(x, y_{-i}, \hat{y}_i\right)  =\varphi_i\left(x, y_{-i}, \lambda y_i^1+(1-\lambda) y_i^2\right) \\
		 		&\qquad\qquad\leq \lambda \varphi_i\left(x, y_{-i}, y_i^1\right)+(1-\lambda) \varphi_i\left(x, y_{-i}, y_i^2\right)
		 	\end{aligned}
		 \end{equation}
		Because $ y_i^1, y_i^2 \in H_i\left(x, y_{-i}\right)$, we have
		\begin{equation}
			\begin{aligned}\label{convex2}
				& \lambda \varphi_i\left(x, y_{-i}, y_i^1\right) \leq \varphi_i\left(x, y_{-i}, y_{i_0}\right),
				 \forall y_{i_0} \in S_{F, i}(x)
			\end{aligned}
		\end{equation}
		\begin{equation}\label{convex3}
			\begin{aligned}
				& \lambda \varphi_i\left(x, y_{-i}, y_i^2\right) \leq \varphi_i\left(x, y_{-i}, y_{i_0}\right), \forall y_{i_0} \in S_{F, i}(x)
			\end{aligned}
		\end{equation}
		Combining \eqref{convex1}, \eqref{convex2} and \eqref{convex3}, we obtain 
		 \begin{equation}\label{convex4}
			\begin{aligned}
				 \varphi_i\left(x, y_{-i}, \hat{y}_i\right) &\leq \lambda \varphi_i\left(x, y_{-i}, y_{i_0}\right)+(1-\lambda)  \varphi_i\left(x, y_{-i}, y_{i_0}\right) \\
				&=\varphi_i\left(x, y_{-i}, y_{i_0}\right),\forall y_{i_0} \in S_{F, i}(x)
			\end{aligned}
		\end{equation}
			i.e., $\hat{y}_i \in H_i\left(x, y_{-i}\right)$.
			So $H_i\left(x, y_{-i}\right)$ is a convex set, and $B(x,\boldsymbol{y})$ is a convex set because of the definition \eqref{B1}. 
		\item  (\romannumeral 3 )  \textbf{Compactness}: Let $\varphi_i^*$ denote the objective value of the problem \eqref{follower22}. Then the set-valued map $H_i$ is represented as follows.
		\begin{equation}\label{compact1}
			H_i\left(x, y_{-i}\right) \triangleq S_{F, i}(x) \cap\left\{y_i \in Y_i \mid \varphi_i\left(x, y_{-i}, y_i\right) \leq \varphi_i^*\right\}
		\end{equation}
		Based on Lemma \ref{lemmaSF}  (\romannumeral 2 ), $S_{F, i}(x)$ is a compact set. $\left\{y_i \in Y_i \mid \varphi_i\left(x, y_{-i}, y_i\right) \leq \varphi_i^*\right\}$ is a compact set since $\varphi_i\left(x, y_{-i}, y_i\right)$ is continuous  with respect to $y_i$ and $Y_i$ is compact. 
		Thus $H_i\left(x, y_{-i}\right)$ must be a compact set, which completes the proof of Lemma \ref{lemmaB} (\romannumeral 3 ).
	\item  (\romannumeral 4 ) \textbf{Semicontinuity}:  Based on  Definition \ref{marginal}, $H_i$ is a marginal map. We can rewrite the form \eqref{Hi1} of the set-valued map $H_i$ as follows.
	\begin{equation}
		\begin{aligned}\label{Hi2}
			& H_i\left(x, y_{-i}\right) \triangleq\left\{y_i \in S_{F, i}(x,y_{-i}) \mid \forall y_i^1 \in S_{F, i}(x,y_{-i})\right. \\
			& \left.\varphi_i\left(x, y_{-i}, y_i\right) \leq \varphi_i\left(x, y_{-i}, y_i^1\right)\right\}
		\end{aligned}
	\end{equation}
	In fact, $S_{F, i}(x,y_{-i})$ is not dependent on $y_{-i}$, definition \eqref{Hi2} is equivalent to definition \eqref{Hi1}. Based on the Assumption \ref{assumption2} and Lemma \ref{lemmaSF}, the payoff function
	$\varphi_i\left(x, y_{-i}, y_i\right)$ is continuous  with respect to $ (x, y_{-i})$
 and the set-valued map $S_{F, i}$ is continuous  with compact values. According to Lemma \ref{lemma2} (\romannumeral 3 ), $H_i$ is upper semicontinuous with respect to $ (x, y_{-i})$. Then the set-valued map $B$ is upper semicontinuous with respect to $(x,\boldsymbol{y})$ because of the definition \eqref{B1}, which completes the proof of Lemma \ref{lemmaB} (\romannumeral 4 ).

 This completes the proof.
 \end{proof}

Key properties of the set-valued map $B$ in the lower-level game \eqref{follower22} are given by Lemma \ref{lemmaB}. It indicates that $B$ is an upper semicontinuous set-valued map with non-empty compact convex values. Then,  we can obtain the following lemma for the upper-level game \eqref{l11}-\eqref{l13}.
	
\begin{lemma} \label{lemmaleader}
	Suppose Assumptions \ref{assumption1} and \ref{assumption2} hold. 	For the upper-level game \eqref{l11}-\eqref{l13}, there exists at least one best response point $x^*$ such that $x^* \in K$.
\end{lemma}
\begin{proof}
As shown in Assumption \ref{assumption1}, the payoff function of the leader $f(x, \boldsymbol{y})$ is continuous with respect to $(x, \boldsymbol{y})$. The set-valued map $B$ is upper semicontinuous with respect to $(x,\boldsymbol{y})$, which is proved in Lemma \ref{lemmaB} (\romannumeral 4 ). Based on Lemma \ref{lemma2} (\romannumeral 2 ), the marginal function $V(x)$ \eqref{VX} is upper semicontinuous with respect to $x$. The leader's strategy space $X$ is a non-empty, compact set. Thus the marginal function $V(x)$ must have a maximum point $x^*$ on $X$ \cite[Page 451]{kreps2013microeconomic}. 
$x^*$ denotes the best response of the leader, i.e., $\exists x^* \in K$. 
\end{proof}

 Based on Lemma \ref{lemmaleader}, $x^* \in K$ is the best response of the leader in game $\Gamma$. Then we need to prove that there exists at least one fixed point $\boldsymbol{y}^*$ in the lower-level game at a given $x^*$. The fixed point $(x^*,\boldsymbol{y}^*)$ is the equilibrium of  the game $\Gamma$. Finally, we can obtain the following theorem. 
 \begin{theorem}
 	\label{theorem2}  
 	Suppose Assumptions \ref{assumption1} and \ref{assumption2} hold, then there exists at least one equilibrium point $(x^*, \boldsymbol{y}^*)$ in game $\Gamma$.
 \end{theorem}
 
 \begin{proof}
 	Let $x^*$ be the best response of the leader according to Lemma \ref{lemmaleader}. The set-valued map $B$  is turned into 
  \begin{align*}
 	 B:  \prod\nolimits_{i=1}^I Y_{i} \rightrightarrows \prod\nolimits_{i=1}^I Y_{i}, \forall x^* \in \arg \max _{x \in X} V(x)
  \end{align*}
 	The set-valued map $B$ is upper semicontinuous and $B(x^*,\boldsymbol{y})$
 is a nonempty compact convex set, which is proved in Lemma \ref{lemmaB}. The set-valued map $B$ has a
 fixed point $\boldsymbol{y}^*$ by Kakutani's fixed-point theorem (Lemma \ref{lemma3}). According to the definition of game $\Gamma$ \eqref{follower2}, $(x^*,\boldsymbol{y}^*)$ is the equilibrium point. This completes the proof of Theorem \ref{theorem2}.
 \end{proof}
 Theorem \ref{theorem2} provides sufficient conditions for the existence of an equilibrium point in the game $\Gamma$ \eqref{follower2}, e.g., (\romannumeral 1 ) continuity; (\romannumeral 2 ) convexity; (\romannumeral 3 ) compactness. 
 Different from the leader-follower game with DIUs \cite{aghassi2006robust, hu2013existence, wang2024existence, liu2018distributionally, singh2017distributionally, fabiani2023distributionally}, the existence of DDUs leads to strategy spaces $\widehat{\Omega}_i$ of followers that are impacted by the leader's strategy $x$. Thus, the existence of an equilibrium in game $\Gamma$ needs us to prove that the set-valued maps $S_{F,i}$ are continuous and have non-empty, compact, convex values. 

\section{Illustrative Examples}
We consider a game $\Gamma$ with one leader, three followers $i\in N=\{1, 2, 3\}$. For the leader, the strategy is denoted by $x \in X \subseteq \mathbb{R}$.
For each follower $i$,  the strategy is denoted by $y_i \in \left\{Y_i \mid \xi_i(x) \cdot y_i \leq b_i\right\} \subseteq \mathbb{R}$, and there is a first-order DDU $\xi_i(x)$.

The leader's problem is defined as
\begin{subequations}\label{l1}
	\begin{align} \label{leader1}
		& \max f(x, \boldsymbol{y})=a\sum\nolimits_{i \in N} xy_{i}+c \\
		\label{leader2}
		& \text { s.t. } x \in X
	\end{align}
\end{subequations}
where $a=1$, $\boldsymbol{y}=(y_1,y_2,y_3)$, $c=1.5$, 
$X=\{x \mid 0 \leq x \leq 8\}$.

The followers’ problems with DDCCs are defined as
\begin{subequations} \label{f1}
	\begin{align} 
		&\min \varphi_i\left(y_i, y_{-i}, x\right)=e_i x y_i-k_i\sum_{-i \in N,-i \neq i} y_{-i} \\
		&\text { s.t. } y_i \in\left\{Y_i \mid \inf _{PD_{{\xi}_i(x)} \in D_i(x)} {\mathbb{P}}\left[
		\xi_i(x) \cdot y_i \leq b_i
		\right]\geq 1-\alpha_i \right\} 
	\end{align}
\end{subequations}
where $e_i=1$, $k_i=2$, $b_i=20$, $\alpha_i=0.05$, $Y_i=\{y_i \mid 0 \leq y_i \leq 20\}$. The decision-dependent ambiguity sets $D_i(x)$ are defined as \eqref{Di}. The DDCC reformulation method in Theorem \ref{theorem1} is utilized to deal with the DDU $\xi_i(x)$. Let $\gamma_{i,1} / \gamma_{i,2} \leq \alpha_i$ hold, 
then the followers’ problems  are
\begin{subequations}\label{F_ddcc}
	\begin{align} 
		\label{F_ddcc1}
		\min \varphi_i&\left(y_i, y_{-i}, x\right)=e_i x y_i-k_i\sum_{-i \in N,-i \neq i} y_{-i} \\
		\label{F_ddcc2}
		\text { s.t. }& y_i \in\left\{Y_i \mid \mu_i(x) y_i+l_i \sqrt{\Sigma_i y_i^2} \leq b_i\right\} \\
		\label{F_ddcc3}
		& \mu_i(x)=\mu_{i}+d_i x\\
		\label{F_ddcc4}
		&l_i=\sqrt{\gamma_{i,1}}+\sqrt{\left(\frac{1-\alpha_i}{\alpha_i}\right)\left(\gamma_{i,2}-\gamma_{i,1}\right)}
	\end{align}
\end{subequations}
where $d_i=0.1$, $\mu_1=1.2$, $\mu_2=1.5$, $\mu_3=1.6$, $\Sigma_1=0.01, \Sigma_2=0.02, \Sigma_3=0.03$, $\gamma_{i,1}=0.01$, $\gamma_{i,2}=1.01$. \eqref{F_ddcc3} represents the relation between the probability distribution of DDU $\xi_i(x)$ and leader's strategy $x$. The followers' problems \eqref{F_ddcc} at present are second-order cone optimization problems, which are convex. In addition, since the leader makes strategy $x$ before followers, we can solve the followers' problems by using KKT conditions. %
The simulation is implemented on MATLAB 2022b, using the YALMIP toolbox and Gurobi 11.0. The equilibrium point of the game is 
\begin{equation}
	\begin{aligned}
		& x^*=7.778, \ \boldsymbol{y}^*=(8.177,	6.825,	6.305). 
	\end{aligned}
	\nonumber
\end{equation}
This example shows the existence of the equilibrium for leader-follower games with DDCCs.

Note that $\gamma_{i,1}$ and $\gamma_{i,2}$ are risk measurement factors, which measure the leader's confidence in the estimated mean $\mu_i\left(x\right)$ and the estimated variance $\Sigma_i$ \cite{delage2010distributionally}. 
In order to measure the impact of $\gamma_{i,1}$ and $\gamma_{i,2}$ on the game equilibrium, we set $\gamma_{i,1}=0.01\rho$ and $\gamma_{i,2}=1.01\rho$. $\rho \in[1,10]$ denotes the weighting factor of the risk measurement. 
Table \ref{tab:1} shows the equilibrium points $(x^*,\boldsymbol{y}^*)$ with different values of weighting factor $\rho$. 
Fig.\ref{Payoff of players} (a) shows that the payoff of the leader decreases as the $\rho$ increases. Because the leader wants to maximize the payoff function \eqref{leader1}, the leader is at a loss as the $\rho$ increases. 
Fig.\ref{Payoff of players} (b) shows that the payoff of followers decreases as the $\rho$ increases. Because followers want to minimize the payoff function \eqref{F_ddcc1}, they are profitable as the $\rho$ increases.  
According to the definition of decision-dependent ambiguity set $D_i(x)$ \eqref{Di}, the statistical distance between the estimated probability distribution of DDU $\xi_i(x)$ and the reference probability distribution will increase with $\gamma_{i,1}$ and $\gamma_{i,2}$ increasing.
So the leader will be conservative in estimating the strategies $\boldsymbol{y}$ of followers and making strategy $x$ to hedge against the worst probability distribution in $D_i(x)$. 

This will inevitably reduce the profits of the leader. 
This economic phenomenon suggests that the leader can improve profits by proactively controlling DDU risk measurement factors $\gamma_{i,1}$ and $\gamma_{i,2}$.
\begin{table}[t]
	\centering
	\footnotesize
	\caption{Equilibrium points with different $\rho$}
	\label{tab:1}  
	\begin{tabular}{ccc} 
		\hline\hline\noalign{\smallskip}	
		\makecell[c]{Weighting factor $\rho$} & \makecell[c]{Equilibrium points $x^*$ } & \makecell[c]{Equilibrium points $\boldsymbol{y}^*$} \\
		\noalign{\smallskip}\hline\noalign{\smallskip}
		1 & 7.778 & (8.177,	6.825,	6.305) 
		\\
		3 & 7.863 & (7.214,	5.896,	5.351) 
		\\
		
		5 & 7.898 & (6.673,	5.391,	4.846)  
		\\
		
		7 & 7.918 & (6.290,	5.040,	4.501) 
		\\
		
		9 & 7.932 & (5.992,	4.771,	4.240) \\
		
		10 & 7.937 & (5.865, 4.658,	4.130) \\
		\noalign{\smallskip}\hline
	\end{tabular}
\end{table}

\begin{figure}[t]
	\centering
	\subfigure[Payoff of the leader ]{
		\includegraphics[width=4.1cm]{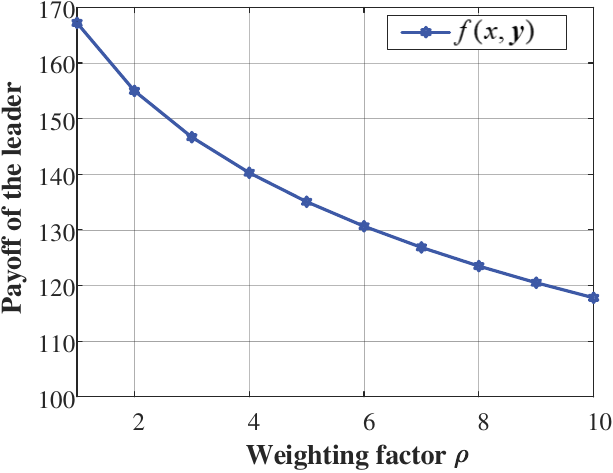}
	}
	\subfigure[Payoff of followers ]{
		\includegraphics[width=4.1cm]{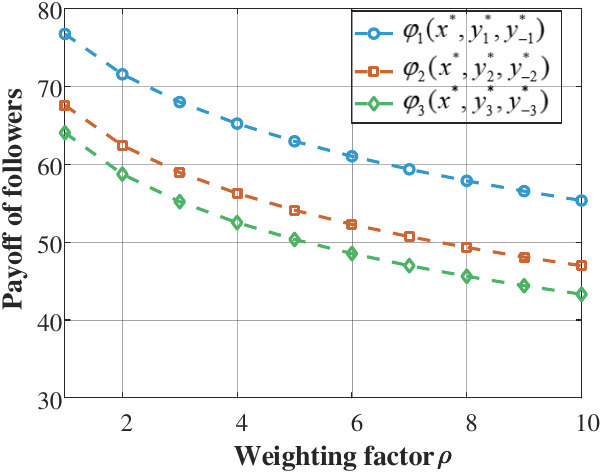}
	}
	\caption{ Payoff of players with different $\rho$}
		\label{Payoff of players}
\end{figure}

\section{Conclusion}
In this paper, we discuss the existence of the equilibrium of the single-leader-multiple-follower game with DDCCs. First, we formulate the game model, where the decision-dependent ambiguity sets are defined based on moment information. The moment information is parameterized by the leader's strategy, which characterizes the probability distributions of DDUs.
Then, the DDCCs are transformed into second-order cone constraints by Cantelli's inequality, which eliminates the probability distributions in the lower-level problem. This allows us to construct set-valued maps with properties of continuity, compactness, and convexity. 
Finally, we prove the existence of the game equilibrium 
based on Kakutani's fixed-point theorem.
Case studies with a one-leader-three-follower game show that the leader can increase profits by controlling the risk measurement factors of DDUs. 

This paper is an extension of the leader-follower game using distributionally robust optimization and also complements the leader-follower game with DDUs.
Since the exact probability distribution information of DDUs is difficult to obtain in real scenarios, the results in this paper are promising in many practical problems, such as Vehicle-to-Grid (V2G) control, multi-stage planning for new energy generation, and disaster management.
The method also enables potential extensions to other forms of game problems with DDCCs.
\appendices
\makeatletter
\@addtoreset{equation}{section}
\@addtoreset{theorem}{section}
\makeatother
\renewcommand{\theequation}{A.\arabic{equation}}
\renewcommand{\thetheorem}{A.\arabic{theorem}}
\setcounter{equation}{0}

\section{Proof of Theorem \ref{theorem1}}
\label{section_appendix_theorem1}
\begin{proof}
	Inspired by the distributionally robust chance constraint method \cite{zhang2018ambiguous}, we first introduce three auxiliary variables $\widetilde{s}_i$, $\widetilde{\beta}_i$, $c_i$ which are defined in \eqref{auxiliary}. Then, three auxiliary sets are defined as follows.
	\begin{align}
		\qquad S&=\left\{\left(\mu_1, \sigma_1\right):\left|\mu_1\right| \leq \sqrt{\gamma_{i,1} \Sigma_i y_{i}^2}, \mu_1^2+\sigma_1^2 \leq \gamma_{i,2} \Sigma_i y_{i}^2\right\} \nonumber\\&
		D_{\widetilde{s}_i}=\left\{\begin{array}{ll}
			&\mathbb{E}_{\mathbb{P}}[\widetilde{s}_i]^2 \leq \gamma_{i, 1} \Sigma_i, \\
			PD_{\widetilde{s}_i} \in \mathcal{P}(\mathbb{R}): &\\
			&\mathbb{E}_{\mathbb{P}}\left[{\widetilde{s}_i}^2\right] \leq \gamma_{i, 2} \Sigma_i \end{array}\right\}\nonumber\\&
		D_{\widetilde{\beta}_i}=\left\{\begin{array}{ll}
			&\left|\mathbb{E}_{\mathbb{P}}[\widetilde{\beta}_i]\right| \leq \sqrt{\gamma_{i,1}\Sigma_i y_{i}^2}, \\
			PD_{\widetilde{\beta}_i} \in \mathcal{P}(\mathbb{R}): &\\
			&\mathbb{E}_{\mathbb{P}}\left[{\widetilde{\beta}_i}^2\right] \leq \gamma_{i,2} \Sigma_i y_{i}^2 
   \end{array}\right\}\nonumber
	\end{align}
where $\mu_1$ and $\sigma_1 ^2$ denote the mean and variance of $\widetilde{\beta}_i$, respectively. Based on the above set, we rewrite the left side of DDCC \eqref{DDCC1} as 
	\begin{equation}
		\begin{aligned}\label{layer1}
			\inf _{PD_{\xi_i(x)} \in D_i(x)}& \mathbb{P}\left(\xi_i(x) \cdot y_i \leq b_i\right) \\ 
			&=\inf _{PD_{\widetilde{s}_i} \in D_{\widetilde{s}_i}} \mathbb{P}\{\widetilde{s}_i y_i \leq c_i\} \\ 
			&=\inf _{PD_{\widetilde{\beta}_i} \in D_{\widetilde{\beta}_i}} \mathbb{P}\{\widetilde{\beta}_i \leq c_i\}\\
			& =\inf_{\left(\mu_1, \sigma_1\right) \in S} \; \inf _{PD_{\widetilde{\beta}_i} \in D_{i,1}\left(\mu_1, \sigma_1^2\right)} \mathbb{P}\{\widetilde{\beta}_i \leq c_i\}
		\end{aligned}
	\end{equation}
Obviously, \eqref{layer1} is a two-level optimization problem. For the outer-level problem, we search for the worst variance and mean value. For the inner-level problem, we search for the worst probability boundary.
	
Applying Lemma \ref{lemma1}, the inner-level problem of \eqref{layer1} is solved as 
	\begin{equation}
		\inf _{PD_{\widetilde{\beta}_i} \in D_{i,1}\left(\mu_1, \sigma_1^2\right)} \mathbb{P}\{\widetilde{\beta}_i\leq c_i\}= \begin{cases}\frac{\left(c_i-\mu_1\right)^2}{\sigma_1^2+\left(c_i-\mu_1\right)^2} & \text { if } c_i \geq \mu_1 \\ 0 & \text { otherwise } \end{cases}
	\end{equation}
	
	\textcolor{black}{ Because the confidence level $1-\alpha_i$ is larger than 0, 
		the two-level problem \eqref{layer1} is equivalent to the following form. }
	\begin{equation}\label{layer2}
		\inf _{PD_{\xi_i(x)} \in D_i(x)} {\mathbb{P}}\left[
		\xi_i(x) \cdot y_i \leq b_i
		\right]=
		\inf _{\left(\mu_1, \sigma_1\right) \in S} \frac{\left(c_i-\mu_1\right)^2}{\sigma_1^2+\left(c_i-\mu_1\right)^2}
	\end{equation}
	
	Recalling the definition of $S$, the following results are obtained by solving problem \eqref{layer2}.
\begin{equation}
    \begin{aligned}\label{layer111}
        &\inf _{PD_{\xi_i(x)} \in D_i(x)} {\mathbb{P}}\left[
        \xi_i(x) \cdot y_i \leq b_i
        \right] \\&= 
        \left\{
        \begin{array}{ll}
            \frac{1}{\left(\frac{\sqrt{\gamma_{i,2}-\gamma_{i,1}}}{\frac{c_i}{\sqrt{\Sigma_i y_i^2}}-\sqrt{\gamma_{i,1}}}\right)^2+1}, & \text{if } \sqrt{\gamma_{i,1}} \leq \frac{c_i}{\sqrt{\Sigma_i y_i^2}} \leq \frac{\gamma_{i,2}}{\sqrt{\gamma_{i,1}}} \\ \\
            \frac{\left(\frac{c_i}{\sqrt{\Sigma_i y_i^2} }\right)^2-\gamma_{i,2}}{\left(\frac{c_i}{\sqrt{\Sigma_i y_i^2}}\right)^2}, & \text{if } \frac{c_i}{\sqrt{\Sigma_i y_i^2}}>\frac{\gamma_{i,2}}{\sqrt{\gamma_{i,1}}} \\ \\
            \text{infeasible}, & \text{if } \frac{c_i}{\sqrt{\Sigma_i y_i^2}}<\sqrt{\gamma_{i,1}}
        \end{array}
        \right.
    \end{aligned}
\end{equation}
	
Based on DDCC \eqref{DDCC1} and results in \eqref{layer111}, we can obtain a second-order cone constraint,
	\begin{equation} \label{eu111}
	\mu_i(x)y_i+l_i \sqrt{\Sigma_i y_i^2} \leq b_i
\end{equation}
where 
\begin{align*}
	l_i=\begin{cases}\sqrt{\gamma_{i,1}}+\sqrt{\left(\frac{1-\alpha_i}{\alpha_i}\right)\left(\gamma_{i,2}-\gamma_{i,1}\right)},
		& \text { when } \gamma_{i,1} / \gamma_{i,2} \leq \alpha_i \\ 
		\sqrt{\frac{\gamma_{i,2}}{\alpha_i}},
		& \text { when } \gamma_{i,1} / \gamma_{i,2}>\alpha_i \end{cases}
\end{align*}

	This completes the proof.
\end{proof}

\section*{References}
\normalem
\bibliographystyle{IEEEtran}

\end{document}